\documentclass[]{amsart}
\usepackage{amssymb}
\usepackage{amsmath, amscd}
\usepackage{latexsym}
\usepackage{epic}
\usepackage{color}
\usepackage{graphicx}
\usepackage{epstopdf}
\usepackage{subfigure}
\numberwithin{equation}{section}
\newtheorem{thm}{Theorem}[section]
\newtheorem{conj}[thm]{Conjecture}

\newtheorem{lem}[thm]{Lemma}

\newtheorem{defn}[thm]{Definition}
\newtheorem{example}{Example}[section]

\begin{document}

%Article information
\title[Graphs with no immersion of large complete graphs]{Constructing graphs with no immersion of large complete graphs}
\date{\today}

% Author information
\author[K. L. Collins]{Karen L. Collins}
\address{Department of Mathematics and Computer Science\\
Wesleyan University\\
Science Tower 655\\
265 Church Street\\
Middletown, CT 06459}
\email[K. L. Collins]{kcollins@wesleyan.edu}
\urladdr[]{http://kcollins.web.wesleyan.edu/}
\author[M. E. Heenehan]{Megan E. Heenehan}
\email[M. E. Heenehan]{mheenehan@wesleyan.edu}
\urladdr[]{http://mheenehan.web.wesleyan.edu/}

% AMS information

\begin{abstract}
In 1989, Lescure and Meyniel proved, for $d=5, 6$, that every $d$-chromatic graph contains an immersion of $K_d$, and in 2003 Abu-Khzam and Langston conjectured that this holds for all $d$. In 2010, DeVos, Kawarabayashi, Mohar, and Okamura proved this conjecture for $d = 7$.
In each proof, the $d$-chromatic assumption was not fully utilized, as the proofs only use the fact that a $d$-critical graph has minimum degree at least $d - 1$.  DeVos, Dvo\v{r}\'ak, Fox, McDonald, Mohar,  and Scheide show the stronger conjecture that a graph with minimum degree $d-1$ has an immersion of $K_d$ fails for $d=10$ and $d\geq 12$ with a finite number of examples for each value of $d$, and small chromatic number relative to $d$, but it is shown that a minimum degree of $200d$ does guarantee an immersion of $K_d$.

In this paper we show that the stronger conjecture is false for $d=8,9,11$ and give infinite families of examples with minimum degree $d-1$ and chromatic number $d-3$ or $d-2$ that do not contain an immersion of $K_d$.  Our examples can be up to $(d-2)$-edge-connected. We show, using Haj\'os' Construction, that there is an infinite class of non-$(d-1)$-colorable graphs that contain an immersion of $K_d$.  We conclude with some open questions, and the conjecture that a graph $G$ with minimum degree $d - 1$ and more than $\frac{|V(G)|}{1+m(d+1)}$ vertices of degree at least $md$ has an immersion of $K_d$.
\end{abstract}

\maketitle

\section{Introduction}\label{intro}

A classic question in graph theory is: if a graph has chromatic number $d$ does the graph ``contain'' a complete graph on $d$ vertices in some way? We know that the containment is not as a subgraph, just consider an odd cycle. In the 1940s Haj\'os~\cite{Ha} conjectured that the necessary containment was subdivision. A graph $H$ is a \textit{subdivision} of a graph $G$, if $G$ contains a subgraph which is isomorphic to a graph that can be obtained from $H$ by subdividing some of the edges. Haj\'os' Conjecture is true for $d \leq 4$ \cite{D}, however, Catlin~\cite{C} showed it is not true for $d \geq 7$ by giving a family of counterexamples. Haj\'os' Conjecture remains open for $d \in \{5, 6\}$.

Another property to consider in trying to answer this question is graph minor. A graph $H$ is a \textit{minor} of a graph $G$ if a graph isomorphic to $H$ can be obtained from a subgraph of $G$ by contracting edges. We will say $G$ has a minor of $H$. In 1943, Hadwiger~\cite{H} conjectured that every loopless $d$-chromatic graph has a minor of $K_d$. Haj\'os' Conjecture is true for $d \leq 4$, so Hadwiger's Conjecture is true for these values. Wagner~\cite{W} showed that the $d = 5$ case of Hadwiger's Conjecture is equivalent to the Four Color Theorem, and Robertson, Seymour, and Thomas \cite{RST} proved the case when $d = 6$; the cases $d \geq 7$ are still open.

Given the difficulties in proving (or disproving) Hadwiger's Conjecture, we choose to explore a different type of containment, namely, graph immersion. The concept of immersion was introduced by Nash-Williams \cite{NW}, as a weakening of graph subdivision, when he conjectured that for every countable sequence $G_i$ $(i = 1, 2, \ldots)$ of graphs, there exist $j > i \geq 1$ such that there is an immersion of $G_i$ in $G_j$. This conjecture was proved by Robertson and Seymour in \cite{RS}. The well-quasi-ordering of graphs by immersion has also been discussed in \cite{G, Ki, L}. Immersions in digraphs have been studied in \cite{CF} and \cite{DM}, and immersion in other contexts can be found in \cite{A, F} and \cite{GK}.
In this paper we will consider immersions of complete graphs.

All graphs in this paper are simple unless otherwise stated. We will follow the notation of West in \cite{We}.

\begin{defn} \cite{AL} A pair of adjacent edges $uv$ and $vw$, $w \neq u$, is \textbf{lifted} by deleting the edges $uv$ and $vw$ and adding the edge $uw$. \end{defn}

\begin{defn} \cite{AL} We say a graph $H$ is \textbf{immersed} in a graph $G$ if and only if a graph isomorphic to $H$ can be obtained from $G$ by lifting pairs of edges and taking a subgraph. If a graph $H$ is immersed in a graph $G$ we may say $G$ has an immersion of $H$.  \end{defn}

It follows from the definition that  if $H$ is immersed in $G$, then the degree of any vertex in $H$ is less than or equal to its degree in $G$, so in order for a $d$-chromatic graph, $G$, to have an immersion of $K_d$, $G$ must have at least $d$ vertices of degree at least $d - 1$.
In 2003, Abu-Khzam and Langston made the following appealing conjecture.

\begin{conj} \cite{AL} \label{main} The complete graph $K_d$ can be immersed in any $d$-chromatic graph.
\end{conj}

Since any $d$-chromatic graph has a $d$-critical subgraph of minimum degree $d - 1$, a proof of the stronger statement that every graph with minimum degree $d - 1$ has an immersion of $K_d$ would imply Conjecture~\ref{main}. Lescure and Meyniel \cite{L}
 used this method to prove the conjecture for $d\leq 6$, and DeVos et al. \cite{De} gave a new proof for $d\leq 7$, but also cite an example of Paul Seymour that shows the stronger statement fails for $d=10$.  Thus, Conjecture \ref{main} remains open for $d\geq 8$.

Following \cite{DeD}, we let $f(d)$ be the smallest integer such that every graph of minimum degree at least $f(d)$ contains an immersion of $K_d$.
Again in \cite{DeD}, the authors show that $f(d)\leq 200d$ for all $d$, and that
$d\leq f(d)$ for $d=10$ and $d\geq 12$.  Their examples are similar to, but more general than, those given by Seymour and give a finite number of examples for each $d$. These examples do not cover the $d = 8, 9$, and $11$ cases, and have small chromatic number relative to $d$.

In this paper we show $f(d) \geq d$ for $d \geq 8 $ by giving infinite families of examples of graphs with minimum degree $d - 1$ and no immersion of $K_d$. Our examples have chromatic number $ d - 3$ or $d - 2$, showing the necessity of some chromatic number bound to get an immersion of $K_d$. Our examples can be up to $(d - 2)$-edge-connected. In addition, we prove that there is an infinite class of $d$-colorable graphs that contain immersions of $K_d$. Finally, we conjecture that if a graph $G$ has minimum degree $d - 1$ and more than $\frac{|V(G)|}{1 + m(d + 1)}$ vertices of degree at least $md$ for any positive integer $m$, then it has an immersion of $K_d$.

The inspiration for our examples comes from the proof of Theorem~\ref{DeVos} given in \cite{De}. We will begin by describing, in general, properties of graphs with minimum degree $d - 1$ and no immersion of $K_d$. We will do this by first considering a graph with an immersion of $K_d$, and exploring where the vertices of that immersion must live within the graph. This will be our Corner Separating Lemma, which appears in Section~\ref{CSL}.  In Section~\ref{dock}, we will describe a general construction of docks, bays, and pods that form graphs with minimum degree $d-1$ and no immersion of $K_d$.  In Section~\ref{examples} we construct appropriate pods, bays, and docks. In Section~\ref{compare}, we compare our examples to those in \cite{DeD}. We conclude with some open questions, a proof that there is an infinite class of examples that satisfy Conjecture~\ref{main}, and a conjecture of our own in Section \ref{Conc}.

\section{Corner Separating Lemma}\label{CSL}

We will begin with a lemma that is useful for determining, given an immersion, where the vertices of that immersion are located in the graph. For our purposes it will be useful to consider an equivalent definition of immersion.

\begin{defn} \cite{AL} A graph $H$ is \textbf{immersed} in a graph $G$ if and only if there exists an injection $\phi: V(H) \rightarrow V(G)$ that can be extended to an injection $\phi_E: E(H) \rightarrow \{\textrm{paths in G}\}$ such that if $u, v \in V(H)$ and $e = uv \in E(H)$ then $\phi_E(e)$ is a path between $\phi(u)$ and $\phi(v)$, and for all $e_1 \neq e_2$, $\phi_E(e_1)$ and $\phi_E(e_2)$ are edge disjoint. \end{defn}

\begin{defn} In an immersion we call image vertices under the injection \textbf{corner} vertices. We call vertices that are on these paths, that are not endpoints of the path, \textbf{pegs}.\end{defn}

We will argue that all corner vertices of an immersion of $K_d$ would have to be in one part of the graph.

\begin{lem}[Corner Separating Lemma]\label{lemma} Let $G$ be a graph and $M$ a subgraph such that there is a cutset of edges $C$ in $G$, $|C| \leq d - 2$, and $M$ is a connected component of $G - C$. If $G$ has an immersion of $K_d$, then all of the corner vertices are in $V(M)$ or all of the corner vertices are in $V(G - M)$.
\end{lem}

\begin{proof}
Let $G$ and $M$ be as described in Lemma~\ref{lemma} and suppose $G$ has an immersion of $K_d$. Suppose for a contradiction that there are corner vertices in both $V(M)$ and $V(G - M)$. Then these corner vertices must be connected by edge disjoint paths, and any path from $V(M)$ to $V(G - M)$ uses an edge of $C$. There are $d$ corner vertices partitioned between $V(M)$ and $V(G - M)$. If there are $x$ corner vertices in $V(M)$, then there are $x(d - x)$ edge disjoint paths between corners in $M$ and corners in $G - M$, where $1 \leq x \leq d - 1$. So, we must have $x(d - x) \leq d - 2$. That is,
\begin{align*}
-x^2 + dx + 2 - d & \leq 0.
\end{align*}
One can check that $-x^2 + dx + 2 - d$ is positive on $1 \leq x \leq d - 1$ giving a contradiction. Therefore all of the corner vertices must be in $V(M)$ or all in $V(G - M)$.
\end{proof}

The following lemmas will also be useful.

\begin{lem}\cite{DeD} \label{edge}If $G$ has an immersion of $K_d$ on a set of $J$ corners, then $G$ has an immersion of $K_d$ on $J$ in which the edges between adjacent vertices in $J$ are used as the paths between these vertices.\end{lem}

This lemma tells us that if there is an immersion of $K_d$, then there is an immersion of $K_d$ that uses edges between adjacent corners.

\begin{lem}\label{NoPeg} Suppose $G$ has an immersion of $K_d$. If a corner vertex of the immersion has degree at most $d$, then it cannot also be used as a peg in the immersion.\end{lem}

\begin{proof} Let $G$ be a graph that has an immersion of $K_d$ and let $v$ be a corner vertex in this immersion with degree at most $d$. Since $v$ is a corner in an immersed $K_d$ the degree of $v$ is either $d - 1\textrm{ or } d$. There are $d$ corners in the immersion so the immersion uses $d - 1$ of the edges incident with $v$ as the paths between $v$ and the other corner vertices. Then there is at most one more edge incident with $v$, so $v$ cannot be a peg.
\end{proof}

The Corner Separating Lemma tells us that if $G$ has an immersion of $K_d$, then all the corners must be in a maximally $(d - 1)$-edge-connected subgraph $G$. In the next section, we will construct graphs so that there can be no immersion of $K_d$ with all of the corners in such a subgraph. 

\section{Docks, Bays, and Pods}\label{dock}

In general we would like to be able to construct graphs with minimum degree $d - 1$ and no immersion of $K_d$. Our construction will be formed by docks, bays, and pods defined below. We will see that as long as we can find docks, bays, and pods with the desired characteristics we can create graphs of minimum degree $d - 1$ with no immersion of $K_d$. The hardest part will be to find pods that satisfy the definition. In \cite{De} the authors prove that for $d \leq 7$ a graph with minimum degree $d$ has an immersion of $K_d$. So we will see it is impossible to find pods for $d \leq 7$.

In each of the following definitions let $d$ be a positive integer.

\begin{defn} A \textbf{$d$-pod} is a graph in which every vertex has degree at least $d - 2$ and no more than $d - 2$ vertices have degree exactly $d - 2$. In addition, there is no immersion of $K_d$ in a $d$-pod, even if a maximum matching of the vertices of degree $d - 2$ is added.\end{defn}

Note that, adding a matching between the vertices of degree $d - 2$ may create multiple edges. When considering an immersion in a larger graph this matching in the pod will represent paths that can be lifted outside the pod to create more connections between vertices in the pod.

\begin{defn} A \textbf{$d$-bay} is a graph with at most $d - 2$ vertices. \end{defn}

\begin{defn} A \textbf{$d$-dock} is composed of one or more $d$-bays arranged in a circle with no more than $d - 3$ edges between any two consecutive $d$-bays, and no edges between nonconsecutive $d$-bays.\end{defn}

Note that, a $d$-dock with two bays can have $2(d - 3)$ edges between the bays.

\begin{defn}A $d$-pod is \textbf{added} to a $d$-bay by adding an edge from each vertex of degree $d - 2$ in the $d$-pod to a vertex in the $d$-bay.  \end{defn}

\begin{defn} A $d$-bay in a $d$-dock is \textbf{full} if $d$-pods are added in such a way that every vertex in the $d$-bay has degree at least $d - 1$. \end{defn}

Note that, a $d$-pod is connected to exactly one bay. When there is no confusion we will drop the $d$ prefix in the terms.

\begin{thm}\label{ours} Let $G_d$ be a $d$-dock in which every $d$-bay is full. Then $G_d$ has minimum degree $d - 1$ and no immersion of $K_d$. \end{thm}

\begin{proof}
Let $G_d$ be a dock in which every bay is full. Let $v \in V(G_d)$. Then $v$ is in a pod or a bay. If $v$ is in a pod, then by the definition of a pod it has degree greater than or equal to $d - 2$ within the pod. If $v$ has degree $d - 2$ within the pod, then by the definition of adding a pod to a bay, there is an edge from $v$ to a bay, so $v$ has degree $d - 1$ in $G_d$. Thus, if $v$ is in a pod it has degree at least $d - 1$. If $v$ is in a bay, then by the definition of a bay being full it has degree at least $d - 1$. Therefore, $G_d$ has minimum degree $d - 1$.

Suppose $G_d$ has an immersion of $K_d$. Then by multiple applications of the Corner Separating Lemma all of the corner vertices are either in a single pod or in the dock. We can use the Corner Separating Lemma because there are at most $d - 2$ edges between any pod and the dock and there are no edges between pods.

The corners cannot all be in a pod because by definition pods have no immersion of $K_d$. The edges out of a pod may provide paths that can be lifted to give a maximum matching of vertices of degree $d - 2$ in the pod, but by definition of a pod even this is not enough to have an immersion of $K_d$ with all of the corners in the pod. Thus all of the corners must be in the dock.

If the dock has fewer than $d$ vertices than we are done, so assume the dock as at least $d$ vertices. If all of the corners are in the dock, then notice they cannot all be in a single bay because the bays have at most $d - 2$ vertices. Thus, the $d$ corners are split between at least two bays. We must consider two cases:
\begin{enumerate}
\item there is a bay with $k$ corners where $2 \leq k \leq d - 2$, or
\item there is at most one corner per bay.
\end{enumerate}

In Case 1, let $B$ be a bay with $k$ corners where $2 \leq k \leq d - 2$. There are at most $2(d - 3)$ edges from $B$ to neighboring bays. So, to get edge disjoint paths from the $k$ corners to the remaining corners in the dock we would need
\begin{align*}
k(d - k) & \leq 2(d - 3)\\
%dk - k^2 & \leq 2(d - 3)\\
-k^2 + dk - 2d + 6 & \leq 0.
\end{align*}
One can check that, the function $-k^2 + dk - 2d + 6$ is positive on the entire interval we are considering giving a contradiction.
%At $k = 2$
%\begin{align*}
%-k^2 + dk - 2d + 6 & = -(2)^2 + 2d - 2d + 6\\
%& = 2 > 0.
%\end{align*}
%At $k = d - 2$
%\begin{align*}
%-k^2 + dk - 2d + 6 & = -(d - 2)^2 + d(d - 2) - 2d + 6\\
%& = -d^2 + 4d - 4 + d^2 - 4d + 6\\
%& = 2 > 0.
%\end{align*}
Therefore, there are not enough edge disjoint paths, and hence no immersion of $K_d$ when there are $k$ corners, $2 \leq k \leq d - 2$, in a bay.

In Case 2, let $B_1$ be a bay with one corner. Let $B_n$ be the next bay in the clockwise direction containing a corner. Let $H$ be the subgraph of $G$ induced by the union of the $B_i$, $1 \leq i \leq n$. Then there are at most $2(d - 3)$ edges connecting corners in $H$ to corners in $G_d - H$. There are 2 corners in $H$ so for there to be an immersion of $K_d$ there must be at least $2(d - 2)$ edge disjoint paths from $H$ to $G_d - H,$ so we need $2(d - 2) \leq 2(d - 3)$, a contradiction. Thus there is no immersion of $K_d$ with all of its corners in the dock and so there is no immersion of $K_d$ in $G_d$.
\end{proof}

We now give a general construction for a $d$-pod with $d + 1$ vertices for $d \geq 8$. 

\begin{defn}\label{P} Let $P$ be a simple graph with $d + 1$ vertices and minimum degree $d - 2$ with at most $d - 2$ vertices of degree exactly $d - 2$. Split the vertices into two sets, $A$ and $B$. Where $A = \{v \in V(P) : \textrm{deg}(v) = d - 2\}$ and $B = V(P) - A$. A \textbf{gadget} is
\begin{enumerate}
  \item a missing odd cycle in $A$, or
  \item a missing path of length 2 with end vertices in $B$ and middle vertex in $A$.
\end{enumerate}
\end{defn}

\begin{thm} \label{gadget}If $P$, as defined in \ref{P}, has three or more gadgets, then $P$ is a $d$-pod.
\end{thm}

\begin{proof}
Let $P$ be as described with at least three gadgets. Since $P$ is simple and $|V(G)| = d + 1$, $d - 2 \leq \textrm{deg}(v) \leq d$ for all $v \in V(G)$. Since $|A| \leq d - 2$, $P$ satisfies the degree requirements to be a $d$-pod. Add a maximum matching to the vertices in $A$. Suppose this new graph, $P^+$, has an immersion of $K_d$. 

Notice that, since the gadgets are missing odd cycles, or have missing edges between $A$ and $B$, the addition of a matching to $A$ leaves at least one vertex in each gadget incident with two missing edges. There are at least three gadgets, so there are at least three vertices each incident with two distinct missing edges. Call these vertices $x, y,$ and $z$. By Lemmas~\ref{edge} and \ref{NoPeg}, $P^+$ has exactly one peg and it is not a corner, call it $w$.

Suppose $w \in B$, then there are at least two gadgets of which $w$ is not a part. Without loss of generality say $w$ is not part of the gadget involving $x$. Then $w$ must be used on edge disjoint paths to replace both missing edges incident with $x$. However, there is at most one path, in fact an edge, from $w$ to $x$ that is not already used in the immersion. Thus $w$ can be used to replace at most 1 edge incident with $x$. So $w \not\in B$.

Thus, $w \in A$. Then there are at least two gadgets of which $w$ is not a part. Without loss of generality say $w$ is not part of the gadgets containing $x$ and $y$. For there to be an immersion of $K_d$ we must use edge disjoint paths through the corner $w$ to replace the two missing edges incident with $x$ and the two missing edges incident with $y$. Without using the matching, there is one unused edge in the graph from $w$ to $x$ and one unused edge from $w$ to $y$. Thus, to replace all four missing edges there must be a matching edge from $w$ to $x$ and a matching edge from $w$ to $y$, a contradiction. There is at most one matching edge incident with $w$. Therefore, there is no immersion of $K_d$ in $P^+$, that is $P$ is a $d$-pod.
\end{proof}

We have given general constructions for graphs with minimum degree $d - 1$ and no immersion of $K_d$. In the next section we give specific examples and explore some of the characteristics of these examples.

\section{Examples}\label{examples}

In 2010, DeVos et al. proved Conjecture~\ref{main} for $d \in \{5, 6, 7\}$ by proving the following theorem.
\begin{thm} \cite{De} Let $f(d)$ be the smallest integer such that every graph of minimum degree at least $f(d)$ contains an immersion of $K_d$. Then $f(d) = d - 1$ for $d \in \{5, 6, 7\}$.\label{DeVos}\end{thm}
In the same paper they report a personal communication of an example given by Paul Seymour that shows $f(d) \geq d$ for every $d \geq 10$. Seymour's example for $d = 10$ is: let $G$ be the graph obtained from $K_{12}$ by deleting the edges of four disjoint triangles. Then $G$ has minimum degree nine, but does not contain an immersion of $K_{10}$. This graph has chromatic number four and contains $K_4$ as a subgraph so is not a counterexample to Conjecture~\ref{main}. In \cite{DeD} DeVos et al. provide a general class of examples showing that $f(d) \geq d$ when $d = 10$ or $d \geq 12$. These examples are similar to, but more general than, those given by Seymour and give a finite number of examples for each $d$. However, these examples do not cover the $d = 8, 9$, and $11$ cases. 

In this section we will give examples of pods and graphs for $d \geq 8$. 
We will give three types of examples. In Example~\ref{finiteEx} we give examples with exactly one bay. The inspiration for these graphs comes from the proof of Theorem~\ref{DeVos}. To prove Theorem~\ref{DeVos} the authors prove the following stronger statement. Note that, when two vertices have multiple edges between them we call the set of edges joining them a \textit{proper parallel class}. 

\begin{thm}\cite{De} \label{DeVos2}Let $d \in \{4, 5, 6\}$, let $G = (V, E)$ be a loopless graph and let $u \in V$. Assume further that $G$ satisfies the following properties:
\begin{itemize}
  \item $|V| \geq d$.
  \item deg$(v) \geq d$ for every $v \in V \backslash \{u\}$.
  \item There are at most $d - 2$ proper parallel classes, and every edge in such a parallel class is incident with $u$.
\end{itemize}
Then there is an immersion of $K_{d + 1}$ in $G$.
\end{thm}

They prove this theorem by supposing there is a minimal (in terms of vertices and edges) counterexample, $G$. They prove properties about $G$ which lead to no such graph existing. In their proof they have a vertex $u$ that may have degree smaller than $d$ and whose neighbors form a complete graph on three vertices. In our initial graphs the pods will be similar to $G$ without $u$. We begin by exploring the case where $d = 8$. In giving our example for $d = 8$ we are showing that Theorems~\ref{DeVos} and \ref{DeVos2} cannot be extended to $d = 8$. We then give examples for each $d \geq 8$. These examples will have chromatic number $d - 3$ or $d - 2$.

In Example~\ref{infinite} we give an infinite number of examples for each $d$ by increasing the number of bays in the graph. These graphs will have chromatic number $d - 2$ and will be 3-edge-connected.

In Example~\ref{LargeConn} we will give examples with edge-connectivity up to $d - 2$ and chromatic number $d - 2$.

\begin{example}{Examples with One Bay}\label{finiteEx}\end{example}

For the $d = 8$ case we form the following graph, $P_8$. 

\begin{defn} Define $P_8$ as follows. Begin with a $K_9$. Remove three disjoint paths of length 2.\end{defn}

The graph $P_8$ is shown in Figure~\ref{Example8}.

\begin{figure}[htbp]
\centering
\subfigure{
\def\svgwidth{2.25in}
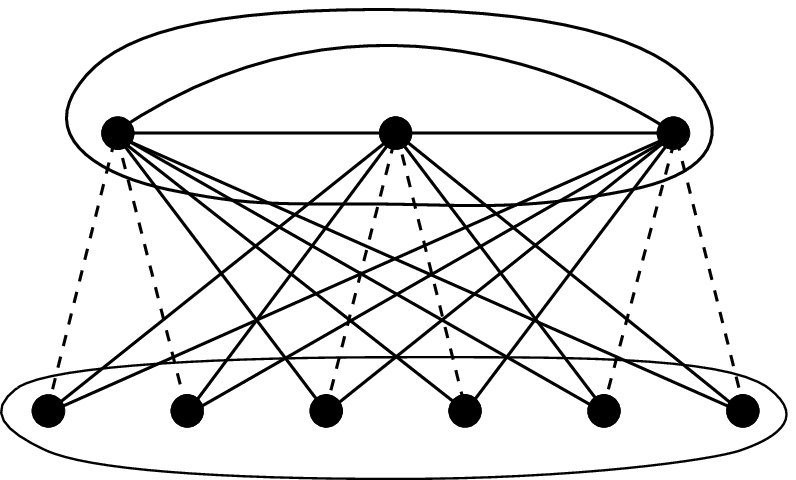
}
\subfigure{
\def\svgwidth{1.75in}
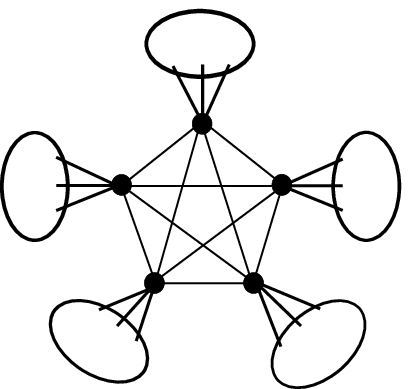
}
\caption[Optional caption for list of figures]{An 8-pod, $P_8$, dotted lines are missing edges, and a possible $G_8$.}
\label{Example8}
\end{figure}

\begin{lem} $P_8$ is an 8-pod.\end{lem}

\begin{proof} 
We can split the vertices into two sets $A$ and $B$. 

Let $A = \{v \in V(P_8)~| \textrm{ deg}(v) = 6\}$ and $B = V(P_8) - A$. Notice that, $|A| = 3 < 6 = d - 2$. The three missing paths of length 2 are gadgets because they have end vertices in $B$ and middle vertex in $A$. Therefore, we can use Theorem~\ref{gadget} to conclude $P_8$ is an 8-pod.
\end{proof}

\begin{thm} \label{G8}Let $G_8$ be the graph with $K_5$ as its only bay, filled with the 8-pods $P_8$. Then $G_8$ has minimum degree 7 and no immersion of $K_8$. \end{thm}

\begin{proof}The $K_5$ is a bay because it has $5 = d - 3 < d - 2$ vertices. The statement follows from Theorem~\ref{ours}. 
\end{proof}

A possible $G_8$ is shown in Figure~\ref{Example8}. Notice that, the 8-pods in this figure are each added to a single vertex giving an example that is 1-vertex connected. Since the bay in $G_8$ is the entire $K_5$ the 8-pods, $P_8$, may be added so that they connect to multiple vertices, and/or more pods may be added, giving different examples still satisfying Theorem~\ref{G8}. The resulting examples could be 2 or 3-vertex-connected. However, all configurations of $G_8$ will be 3-edge-connected. The chromatic number of $G_8$ is 6 and $G_8$ has a subgraph of $K_6$.

We will build all of our examples in a similar way. We must construct $G_d$ and $P_d$ in general. We construct $G_d$ by starting with a $K_{d - 3}$ as the only $d$-bay in a $d$-dock and make this bay full by adding copies of the $d$-pod $P_d$. 

\begin{defn}\label{P_d} For every $d \geq 8$ we construct $P_d$ in the following way. Begin with a $K_{d + 1}$. Remove three disjoint paths of length 2. Remove a maximum matching from the vertices that are not on these paths. \end{defn}

Note that, if $d$ is odd there will be one vertex in $P_d$ of degree $d$. The graph $P_d$, for $d$ odd, is shown in Figure~\ref{GeneralH}, the dotted lines represent missing edges. 

\begin{figure}[htbp]
\centering
\def\svgwidth{4.75in}
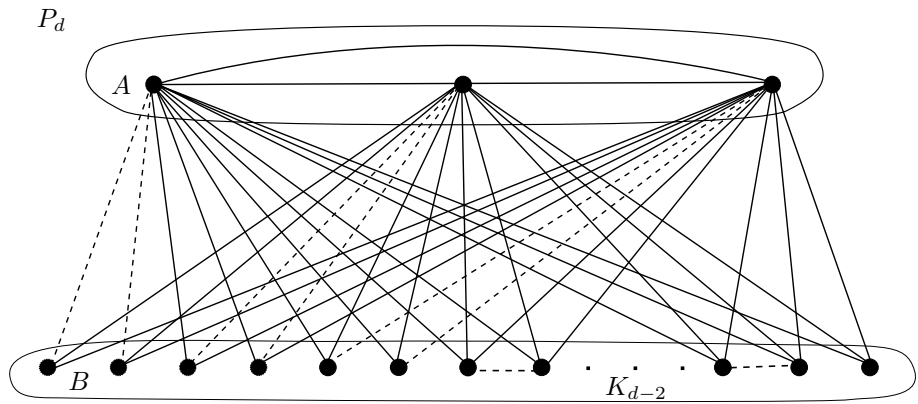
\caption[Optional caption for list of figures]{A $d$-pod, $P_d$, with $d$ odd. Dotted lines are missing edges.}
\label{GeneralH}
\end{figure}

\begin{lem}\label{PdPod} $P_d$ is a $d$-pod. \end{lem}

\begin{proof}
We can split the vertices into two sets $A$ and $B$. 

Let $A = \{v \in V(P_d)~| \textrm{ deg}(v) = d - 2\}$ and let $B = V(P_d) - A$. The only vertices of degree $d - 2$ are those that are the centers of the missing paths, therefore $|A| = 3 < d - 2$ for $d \geq 8$. The three missing paths are gadgets because they have end vertices in $B$ and middle vertex in $A$. Therefore, we can use Theorem \ref{gadget} to conclude $P_d$ is a $d$-pod.
\end{proof}

\begin{thm}\label{Gd}
Let $G_d$ be the graph with $K_{d - 3}$ as its only bay filled with copies of $P_d$. The graph $G_d$ has minimum degree $d - 1$ and no immersion of $K_d$.
\end{thm}

\begin{proof}The statement follows from Theorem~\ref{ours}.\end{proof}

Notice that, for $d =8, 9$, $\chi(G_d) = d - 2$ and $G_d$ has a subgraph of $K_{d - 2}$. For $d \geq 10$, $\chi(G_d) = d - 3$ and $G_d$ has a subgraph of $K_{d - 3}$ and an immersion of $K_{d - 1}$. Therefore $G_d$ is not a counterexample to Conjecture~\ref{main}. Notice that, given the construction of $G_d$ we can create a 1-vertex-connected graph by attaching pods to exactly one vertex in the bay, or 2 or 3-vertex-connected graphs by attaching pods to multiple vertices in the bay, however in all cases we get a 3-edge-connected graph. Next we create examples with multiple bays for each $d \geq 8$.

\begin{example}{Examples with Multiple Bays}\label{infinite}\end{example}

The examples given so far have exactly one bay in a dock. We will now consider several cases where the dock has more than one bay. Let $d \geq 8$ be a fixed integer. We will use the same pods, $P_d$, as in Definition~\ref{P_d} and will use bays labeled $B^0, \ldots, B^{n - 1}$, each bay $B^i$ is isomorphic to $K_{d - 2}$. We form the docks by placing $n$ copies of $K_{d - 2}$, the $B^i$, in a circle. The idea for connecting consecutive bays, $B^i$, is to add edges from half of the vertices in $B^i$ to half of the vertices in the next bay, $B^{i + 1}$, and edges from the other half of the vertices in $B^i$ to half of the vertices in the previous bay, $B^{i - 1}$. The following is a precise description of how to connect consecutive bays. 

\begin{defn}\label{connect bays} Label the vertices of $B^i$ as
$$a^i_1, a^i_2, \ldots, a^i_{d - 2}.$$
\textbf{Connect consecutive bays} by adding edges between $B^i$ and $B^{i - 1}$, where we consider the superscripts $\mod n$,  so that 
$$a^i_j \textrm{ is  adjacent to } a^{i - 1}_{d - 1 - j} \textrm{ for } 0 \leq i \leq n -1, \textrm{ and }1 \leq j \leq \bigg\lceil\frac{d - 2}{2}\bigg\rceil.$$
\end{defn}

%A dock with odd $d$ is shown in Figure~\ref{Dock}.

%\begin{figure}[htb]
%\centering
%\def\svgwidth{5.5in}
%\input{OddDock.pdf_tex}
%\caption[Optional caption for list of figures]{Dock for $d$ odd.}
%\label{Dock}
%\end{figure}

\begin{thm} Let $d \geq8$ and $G^n_d$ be a graph with a $d$-dock formed by connecting the $B^i$ ($1 \leq i \leq n$) as described in Definition~\ref{connect bays}. Each $B^i$ is made full by adding copies of the $d$-pod $P_d$. Then $G^n_d$ has minimum degree $d - 1$ and no immersion of $K_d$. \end{thm}

\begin{proof} We must show that we may apply Theorem~\ref{ours}. We proved in Lemma~\ref{PdPod} that the $P_d$ satisfy the definition of a $d$-pod. The $B^i$ are indeed bays because they are copies of $K_{d - 2}$. The number of edges between consecutive bays is $\lceil\frac{d - 2}{2}\rceil$ and 
$$\bigg\lceil\frac{d - 2}{2}\bigg\rceil \leq \frac{d}{2} \leq d - 3$$
for $d \geq 4$. Now we may apply Theorem~\ref{ours} to determine $G^n_d$ is a graph with minimum degree $d - 1$ and no immersion of $K_d$.
\end{proof}

Notice that, since the bays are copies of $K_{d - 2}$, $\chi(G^n_d) \geq d - 2$. In fact $\chi(G^n_d) = d - 2$. A case where $d = 9$ and $n = 4$, i.e. there are four bays, is shown in Figure~\ref{M8}. 
\begin{figure}[h!]
\centering 
\def\svgwidth{3.25in}
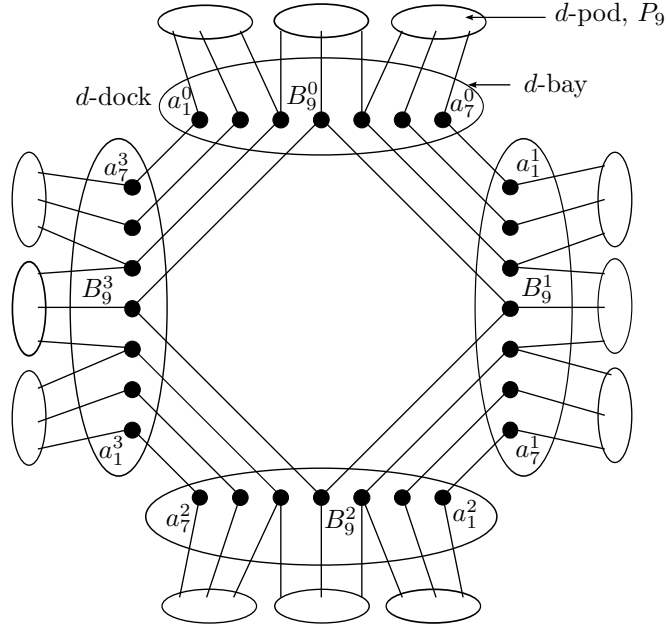
\caption[Optional caption for list of figures]{An example of $G^4_9$ with  3 copies of $P_9$ added to each bay.}
\label{M8}
\end{figure}
Note the $a^i_3, a^i_4, a^i_5$ have degree $d = 9$, but the other vertices in the dock have degree $d - 1 = 8$. This is just one example of $G^4_9$, we could have more (or fewer) bays and/or more pods, $P_9$. The example shown is 3-vertex-connected, but if pods are added in a different way we could have examples with vertex-connectivity 1, 2, or 3. However, the graph $G^n_d$ is always 3-edge-connected. Using the $d$-pods $P_d$ to create a graph $G$ with no immersion of $K_d$ will result in a graph that is at most 3-vertex-connected and exactly 3-edge-connected. Next we create graphs with larger connectivity.

\begin{example}{Examples with Greater Edge-Connectivity}\label{LargeConn}\end{example}

First we give examples of graphs with minimum degree $d - 1$ and no immersion of $K_d$ that can be up to 5-vertex-connected and are exactly 5-edge-connected. To create these examples we will use the same docks as in Definition~\ref{connect bays}, but must modify the pods that are added to the bays.

\begin{defn} For every $d \geq 8$ construct $P^5_d$ in the following way. Begin with a $K_{d + 1}$. Remove two disjoint paths of length 2 and a disjoint 3-cycle. Remove a maximum matching from the vertices not involved in a missing path or the missing 3-cycle.\end{defn}

The superscript of $P^5_d$ indicates the edge-connectivity of the example of which it is a part. The graph $P^5_d$ is shown in Figure~\ref{5Connd} with the edges that connect it to a bay.

\begin{figure}[htb]
\centering
\def\svgwidth{4in}
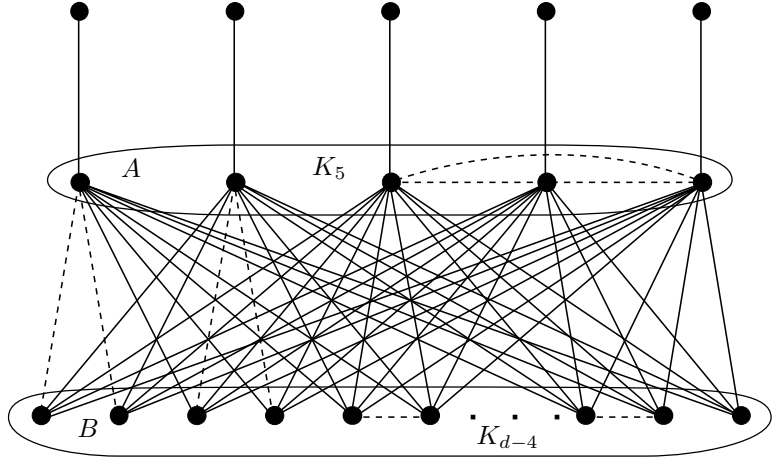
\caption[Optional caption for list of figures]{5-vertex-connected pod, $P^5_d$, with $d$ odd. Missing edges are represented by dotted lines.}
\label{5Connd}
\end{figure}

\begin{lem}\label{5conn}The graph $P^5_d$ is a $d$-pod.\end{lem}

\begin{proof}
Split the vertices into two sets $A$ and $B$, as shown in Figure~\ref{5Connd}. Let $A = \{v \in V(P^5_d)~| \textrm{ deg}(v) = d - 2\}$ and let $B =  V(P^5_d) - A$. The only vertices of degree $d - 2$ are the centers of the two missing paths and the three vertices in the missing 3-cycle. Therefore $|A| = 5 < d - 2$ for $d \geq 8$. The two missing paths of length 2 and the missing 3-cycle are gadgets and $|V(P^5_d)| = d + 1$, so we can use Theorem~\ref{gadget} to conclude $P^5_d$ is a $d$-pod.
\end{proof}

\begin{defn} For $d \geq 8$, construct the graph $H^5_d$ as follows. Use $n \geq 1$ bays isomorphic to $K_{d - 2}$, connect consecutive bays as in Definition~\ref{connect bays}. Make the bays full by adding copies of $P^5_d$. \end{defn}

The graphs $H^5_d$ can be 1, 2, 3, 4, or 5-vertex-connected and are 5-edge-connected. The chromatic number of $H^5_d$ is $d - 2$ and these graphs have subgraphs, and thus an immersion, of $K_{d - 2}$.

\begin{thm} The graph $H^5_d$ has minimum degree $d - 1$ and no immersion of $K_d$.\end{thm}

\begin{proof}The proof is by Theorem~\ref{ours} and Lemma~\ref{5conn}.\end{proof}

A similar idea will give us examples with edge-connectivity $k$ where $d \geq 9$ and $7 \leq k \leq d - 2$. To do this we will give a new way to construct $d$-pods with exactly $k$ vertices of degree $d - 2$, which we will label $P^k_d$.  

\begin{defn}\label{k-conn} Construct $P^k_d$, for $d \geq 9$, $7 \leq k \leq d - 2$ as follows. Begin with a $K_{d + 1}$. We continue with two cases depending on the parity of $k$:

\begin{enumerate}
  \item If $k$ is  an odd integer, then remove a path of length 2, a 3-cycle, and a $C_{k - 4}$. These should be disjoint. Finally, remove a maximum matching of the vertices not in the path, the 3-cycle, or the $C_{k - 4}$.
  
  \item If $k$ is an even integer, then remove three disjoint 3-cycles. If $k - 9 = 1$, remove a path of length 2 disjoint from the 3-cycles. If $k - 9 > 1$ remove a $C_{k - 9}$ disjoint from the 3-cycles. Finally remove a maximum matching of the vertices that have not yet been used.
  \end{enumerate}
\end{defn}

\begin{lem} $P^k_d$ is a $d$-pod.\label{kconn}
\end{lem}

\begin{proof} We must show that we may apply Theorem~\ref{gadget}. By construction $|V(P^k_d)| = d + 1$. The minimum degree of $P^k_d$ is $d - 2$, there are $k$ vertices of degree exactly $d - 2$, and $k \leq d - 2$ by definition. Let $A = \{v \in V(P^k_d) | \textrm{ deg}(v) = d - 2\}$ and $B = V(P^k_d) - A$. By construction $P^k_d$ has three gadgets. Thus by Theorem~\ref{gadget} $P^k_d$ is a $d$-pod.
\end{proof}

Now using the same docks as in Definition~\ref{connect bays} we create examples of graphs, which we will label $M_d^k$, with minimum degree $d - 1$, edge-connectivity $k$ ($7 \leq k \leq d - 2$) and no immersion of $K_d$. 

\begin{defn}\label{M} For $d \geq 9$ form the graph $M^k_d$ using $n \geq 1$ bays that are isomorphic to $K_{d - 2}$, connect consecutive bays as in Definition~\ref{connect bays}. Make the bays full with copies of the $d$-pod $P^k_d$.\end{defn}

\begin{thm} The graph $M^k_d$ has minimum degree $d - 1$ and no immersion of $K_d$. \end{thm}

\begin{proof} The proof is by Theorem~\ref{ours} and Lemma~\ref{kconn}.
\end{proof}

Note that, $M^k_d$ is $k$-edge-connected. In fact Definition~\ref{k-conn} tells us that we can create examples that are $(d - 2)$-edge-connected for $d = 9$ and $d \geq 11$. We give the following special example for $d = 10$ and $G$ is 8-edge-connected.

\vspace{.1in}

\noindent\textbf{Example}
Let $P$ be a simple graph with 11 vertices. Begin with a $K_{11}$ and remove two disjoint 3-cycles, a disjoint edge, and a disjoint path of length 2. Now $P$ has 7 vertices of degree 8 and 4 vertices of degree 9. Let $A = \{v \in V(P)~|~\textrm{deg}(v) = 8\}\cup\{w\}$, where $w$ is one of the vertices of the missing edge. Let $B = V(P) - A$. Use the same docks and bays as those used for the $M^k_d$ defined in \ref{M}. Make the bays full by copies of $P$ where we add one edge from each vertex in $A$ to a bay. Notice that, $P$ is not quite a pod because we are adding edges from a vertex of degree 9 to the bays. Let this new graph be called $G$.

\begin{lem} $G$ has minimum degree 8, is 8-edge-connected and has no immersion of $K_{10}$.\end{lem}

\begin{proof}
$G$ has minimum degree 8 and is 8-edge-connected by construction. 

Suppose $G$ has an immersion of $K_{10}$. Using the Corner Separating Lemma we see that all of the corner vertices would be in the dock or in a single copy of $P$. We know from Theorem~\ref{ours} that all of the corners cannot be in the dock. Suppose all of the corners are in a copy of $P$. Then there is exactly one peg, label the peg $x$. Suppose $x \in A$. The edges out of $P$ can be used to replace at most one edge in each of the missing 3-cycles. Therefore, there are at least two vertices, say $y$ and $z$, in $A$ distinct from $x$ that are incident with two missing edges. We know $xy, xz \in E(G)$. To replace all four missing edges we must have another path from $x$ to $y$ and an edge disjoint path from $x$ to $z$. This can only be accomplished using the edge out of $P$ that is incident with $x$, but there is only one such edge. Therefore $x \not\in A$. 

Suppose $x \in B$. Then again there is at least one vertex in $A$, say $y$, that is incident with two missing edges, each of which must be replaced by an edge disjoint path through $x$. However, there is only one unused edge incident with $y$, so at most one of these edges can be replaced, i.e. $x \not\in B$. Thus there is no immersion of $K_{10}$ in $G$. 
\end{proof}

We have now given examples of graphs that have minimum degree $d - 1$ that are $(d - 2)$-edge-connected, have chromatic number $d - 2$, and have no immersion of $K_d$ for $d \geq 8$. In the next section we compare our examples with those given in \cite{DeD}.

\section{Comparison of Examples}\label{compare}

We now compare our examples with those given in \cite{DeD} where the authors prove the following theorem.

\begin{thm}\cite{DeD} \label{DeEx}Suppose $H_1, \ldots, H_t$ are simple $D$-regular graphs, each with chromatic index $D + 1$, where $t > \frac{1}{2}D(D + 1)$. Let $G$ be the complement of the graph formed by taking the disjoint union of $H_1, \ldots, H_t$. Letting $n$ denote the number of vertices of $G$, the minimum degree of $G$ is $n - 1 - D$, but $G$ does not contain an immersion of the complete graph on $n - D$ vertices. \end{thm}

Theorem~\ref{DeEx} gives examples that show for $d = 10$ and $d \geq 12$ there are graphs with minimum degree $d - 1$ that contain no immersion of $K_d$. We give examples for $d \geq 8$.

\begin{lem}\label{finite}For a fixed $d$ there are a finite number of examples of the type in Theorem~\ref{DeEx}.\end{lem}

\begin{proof}
Let $G$ be a graph of the type described in Theorem~\ref{DeEx}, then $d = n - D$ and $n = \sum_{i = 1}^t|V(H_i)|$. Since the smallest $D$-regular graph with chromatic index $D + 1$ has at least $D + 1$ vertices we get the following,
\begin{align}
	n = \sum_{i = 1}^t|V(H_i)| &\geq (D + 1)t\\
	n = d + D &\geq (D + 1)\bigg(\frac{1}{2}D(D + 1) + 1\bigg)\\
%	d + D &\geq \frac{1}{2}D^3 + D^2 + \frac{3}{2}D + 1\\
	d &\geq \frac{1}{2}D^3 + D^2 + \frac{1}{2}D + 1.\label{inequal}
\end{align}
	
For fixed $d$ there are only a finite number of solutions (of positive integers) for $D$. Therefore, given $d$ there are only a finite number of examples of the type described in Theorem~\ref{DeEx}.
\end{proof}

For fixed $d$, our constructions allow for an infinite number of examples.

The proof of Lemma~\ref{finite} helps us to see that Theorem~\ref{DeEx} gives examples for $d = 10$ and $d \geq 12$. Notice that, if $d = 11$ Inequality~(\ref{inequal}) gives $10 \geq \frac{1}{2}D^3 + D^2 + \frac{1}{2}D$. The only positive integer for which this is true is $D = 2$. When $D = 2$, $t \geq 4$, and $n = 13$. The smallest 2-regular graph with chromatic index 3 is $K_3$ which has 3 vertices. This means for there to be an example $t = 4$. However, there are not four 2-regular graphs with chromatic index 3 that would give a graph with 13 vertices. Thus there is no example of the type described in Theorem~\ref{DeEx} for $d = 11$. For $d = 10$ and $d \geq 12$ examples can be formed using $D = 2$, $t = 4$, and each $H_i$ some odd cycle.

We noted that the chromatic number of our graphs is $d - 2$, and our graphs have subgraphs, and thus immersions, of $K_{d - 2}$. Let us consider the chromatic number of the graphs described in Theorem~\ref{DeEx}. We know that for each $H_i$, $\chi(H_i) \leq D + 1$ because each $H_i$ is $D$-regular. Therefore, $\Bar{H_i}$ (the complement of $H_i$) is $[|V(H_i)| - (D+1)]$-regular and 
$$\chi(\bar{H_i}) \leq |V(H_i)| - (D + 1) + 1 = |V(H_i)| - D.$$
Since $G$ is the complement of the disjoint union of the $H_i$ we have,
\begin{align*}
    \chi(G) = \sum\chi(\bar{H_i}) & \leq \sum(|V(H_i)| - D)   \\
    & \leq n - tD\\
    & \leq n - D - (t - 1)D
\end{align*}

In the language of these graphs, our graphs have chromatic number $n - D - 2$.

\begin{align*}
(t - 1)D & \geq \bigg(\frac{1}{2}D(D + 1) - 1\bigg)D\\
& \geq \bigg(\frac{1}{2}D^2 + \frac{1}{2}D - 1\bigg)D\\
& \geq \frac{1}{2}D^3 + \frac{1}{2}D^2 - D\\
& > 2 \textrm{ for } D \geq 2.
\end{align*}

Since $D \geq 2$ in all of the graphs in Theorem~\ref{DeEx} the chromatic number of these graphs is smaller than the chromatic number of our graphs. Conjecture~\ref{main} is that every $d$-chromatic graph contains an immersion of $K_d$. Our examples show that chromatic number $d - 2$ is not large enough to give an immersion of $K_d$.

\section{Conclusion and Open Questions}\label{Conc}

In this paper we showed for $d \geq 8$ a graph with minimum degree $d - 1$ need not have an immersion of $K_d$. This adds to the previous work by Lescure and Meyniel in \cite{L} and that done by DeVos et al. in \cite{De} by settling the cases for $d = 8$ and $d = 9$. We also gave infinite families of graphs with minimum degree $d - 1$ and no immersion of $K_d$ that are different than those given in \cite{DeD}. The examples that we give have chromatic number $d - 2$ and while they do not have an immersion of $K_d$, they have a subgraph of $K_{d - 2}$. In creating this family of graphs we realized that connectivity plays a key role in this question. We were able to find graphs with minimum degree $d - 1$ and edge-connectivity $d - 2$, for $d \geq 9$, with no immersion of $K_d$. Given that our Corner Separating Lemma relies on there being at most $d - 2$ edges between different parts of the graph a different approach would be needed to give examples with edge-connectivity greater than $d - 2$. This has led us to ask the following questions.

\begin{enumerate}
  \item Do graphs with large connectivity have to have an immersion of a large complete graph? This would not help in proving the conjecture of Abu-Khzam and Langston, since graphs with large chromatic number may have small connectivity, but might shed some light on the structure necessary to have an immersion of a large complete graph.
  \item Are there $d$-pods with more than $d + 1$ vertices and what do they look like?
  \item The chromatic number of our examples is $d - 2$. Are there examples of graphs with minimum degree $d - 1$ and chromatic number $d - 1$ with no immersion of $K_d$?
\end{enumerate}

While we have not resolved more cases of Conjecture~\ref{main}, we note that there is an infinite class of graphs satisfying Conjecture~\ref{main}. This can be seen by considering Haj\'os's Construction~\cite{Ha}: The following set of operations on simple graphs produce non-$k$-colorable graphs from non-$k$-colorable graphs, and in fact every non-$k$-colorable graph can be constructed by beginning with a $K_{k +1}$ and repeating these operations.
\begin{enumerate}
  \item[($\alpha$)]  Addition of edges and/or vertices to the graph.
  \item[($\beta$)] Identification of two non-adjacent vertices and deletion of the resulting multiple edges.
  \item[($\gamma$)] For two graphs $G_1$ and $G_2$ and $x_iy_i \in E(G_i)$, deletion of $x_1y_1$ and $x_2y_2$, addition of the new edge $y_1y_2$, and identification of the vertices $x_1$ and $x_2$.
\end{enumerate}

The above statement and a proof appear in \cite{La}.

Given this construction a possible approach to proving Conjecture~\ref{main} is to show that immersion is preserved by the operations. 

\begin{lem} Immersions of $K_{d}$ are preserved by applications of $(\alpha)$ and $(\gamma)$.\end{lem}

\begin{proof}
Let $G$ be a graph with an immersion of $K_{d}$. Adding edges and/or vertices to this graph does not change the immersion of $K_{d}$. Thus, $(\alpha)$ preserves the immersion.

Let $G_1$ and $G_2$ be graphs, each of which has an immersion of $K_d$. Let $x_1y_1 \in E(G_1)$ and $x_2y_2 \in E(G_2)$. Let $H$ be the graph obtained by applying $(\gamma)$.

If there is a path $P_1$ from $x_1$ to $y_1$ in $G_1 - x_1y_1$, then there is an immersion of $K_d$ in $H$ using the immersion of $K_d$ in $G_2$ and replacing $x_2y_2$ by $P_1 + y_1y_2$, if necessary.

Similarly, if there is a path $P_2$ from $x_2$ to $y_2$ in $G_2 - x_2y_2$, then there is an immersion of $K_d$ in $H$ using the immersion of $K_d$ in $G_1$ and replacing $x_1y_1$ by $P_2 + y_1y_2$, if necessary.

If there is no $P_1$ in $G_1 - x_1y_1$ and no $P_2$ in $G_2 - x_2y_2$, then $x_1y_1$ is a cut-edge in $G_1$ and $x_2y_2$ is a cut-edge in $G_2$. Since these are cut-edges the Corner Separating Lemma tells us that the corners in the immersion of $K_d$ in $G_i$ are all on one side of the graph, i.e. the immersion does not use the edge $x_iy_i$ $(i \in \{1, 2\})$. Thus, the immersions of $K_d$ in $G_i$ are immersions of $K_d$ in $H$ $(i \in \{1, 2\})$.
\end{proof}

This lemma tells us that the class of non-$k$-colorable graphs obtained from applying operations $(\alpha)$ and $(\gamma)$, starting with a $K_{k + 1}$, satisfy Conjecture~\ref{main}. To prove Conjecture~\ref{main}, one would have to prove that immersions of $K_d$ are preserved by the operation $(\beta)$. Proving this seems quite complicated. Since $(\beta)$ is the identification of any two non-adjacent vertices, if two non-adjacent corners are identified we would need to show that another vertex in the new graph could become a corner. It seems that, while applying $(\alpha)$ and $(\gamma)$ result in a graph with a very similar immersion to the original graph (or graphs), applying $(\beta)$ could result in a very different immersion than that in the original graph.

Finally, we conclude with a conjecture. In \cite{DeD} the authors prove

\begin{thm}[\cite{DeD}] Every simple graph with minimum degree at least $200d$ contains a strong immersion of $K_d$.\end{thm}

Where a strong immersion means that corner vertices of the immersion are not also used as pegs in the immersion. This made us question how many vertices in a graph can have degree $md$, for $m$ any positive integer, and still have no immersion of $K_d$? If we construct the graph $M^{d - 2}_d$ where we add pods $P^{d - 2}_d$ to single vertices and we add enough pods to give the vertices in the dock degree at least $md$ for $m$  any positive integer. If $p$ is the number of pods added to each vertex and $b$ is the number of bays in the dock, then $$|V(M^{d - 2}_d)| = (d - 2)b(1 + (d + 1)p).$$ So, we can create graphs with $$\frac{|V(M^{d - 2}_d)|}{1 + p(d + 1)}$$ vertices of degree $md$ and no immersion of $K_d$. This has led us to make the following conjecture.

\begin{conj} Let $m$ be a positive integer, and $d \geq 8$ large compared to $m$. If a graph $G$ has minimum degree $d - 1$ and more than $\frac{|V(G)|}{1 + m(d + 1)}$ vertices have degree at least $md$, then $G$ has an immersion of $K_d$.\end{conj}

\end{document}